\documentclass{article}

\usepackage[margin=1.2in]{geometry}
\usepackage{amsmath,amssymb,amsthm,mathtools,mathrsfs,bm}
\usepackage{graphicx}
\usepackage{xcolor}
\usepackage[shortlabels]{enumitem}
\usepackage{tikz}
\usetikzlibrary{calc,shapes,backgrounds}
\usepackage[colorlinks=true,linkcolor=blue,citecolor=blue,urlcolor=blue]{hyperref}

\allowdisplaybreaks

\newtheorem*{quotedtheorem}{Theorem}
\newtheorem{theorem}{Theorem}[section]

\newtheorem{lemma}[theorem]{Lemma}

\newtheorem{claim}[theorem]{Claim}

\DeclareMathOperator{\ex}{ex}
\title{Codegree-Constrained Turán Problems for $F_5$
 and Its Blow-ups}
\date{\today}
\author{
Yupei Li\thanks{University of South Carolina, Columbia, SC 29208, USA. Email: \texttt{yupei@email.sc.edu}.}
\and
 Linyuan Lu\thanks{University of South Carolina, Columbia, SC 29208,
({\tt lu@math.sc.edu}).
This author was supported in part by NSF grants DMS 2038080 and DMS 2601952.
}
}
\begin{document}

\maketitle

\begin{abstract}
Let $F_5$ denote the $3$-uniform hypergraph with edge set
$\{abc,abd,cde\}$. For a hypergraph property $\mathcal P$, let
$\ex_{\mathcal P}(n,F_5)$ be the maximum number of edges in an
$n$-vertex $F_5$-free $3$-uniform hypergraph satisfying $\mathcal P$. Let $\Delta_2(H)$ denote the maximum codegree of a pair of vertices in a $3$-uniform hypergraph $H$.
We determine the exact value of $\ex_{\Delta_2\ge m}(n,F_5)$ when
$m$ is bounded away from $n/3$ by a fixed positive proportion. More
precisely, for every constant $c\in (1/3,1)$, there exists $n_0$ such that
for all $n\ge n_0$ and $m=cn$,
\[
\ex_{\Delta_2\ge m}(n,F_5)
=
m\left\lceil\frac{n-m}{2}\right\rceil
\left\lfloor\frac{n-m}{2}\right\rfloor.
\]
The unique extremal hypergraph is the complete tripartite $3$-graph
$K^3_{m,\lceil (n-m)/2\rceil,\lfloor (n-m)/2\rfloor}$.
We also prove the corresponding exact result for the blow-ups $F_5^t$,
together with stability theorems showing that every near-extremal
configuration is close to the relevant tripartite construction.
\end{abstract}
\section{Introduction}
Given a family $\mathcal{F}$ of graphs, a graph $G$ is \emph{$\mathcal{F}$-free} if it contains no member of $\mathcal{F}$ as a subgraph. The classical Turán number $\ex(n,\mathcal{F})$ denotes the maximum number of edges in an $\mathcal{F}$-free graph on $n$ vertices.

Recently, several authors have investigated Turán-type problems under additional structural constraints. In particular, it has been observed that if an $\mathcal{F}$-free graph is required to contain a large prescribed substructure, then its maximum number of edges can be strictly smaller than the classical Turán bound.

A pioneering result in this direction was obtained by Balister, Bollobás, Riordan, and Schelp~\cite{BBRS}, who proved that every $C_{2k+1}$-free graph on $n$ vertices with maximum degree at least $m>\lceil n/2\rceil$ has at most $m(n-m)$ edges, with equality attained by the complete bipartite graph $K_{m,n-m}$. Huo and Yuan~\cite{HY} later generalized this result to all edge-critical graphs. Several related questions have been investigated recently. See~\cite{LRW} for the effect of a large-degree vertex on the Andr\'asfai--Erd\H{o}s--S\'os theorem, and see~\cite{HHLLYZ} for analogous questions on Tur\'an problems for degenerate hypergraphs under a large-degree condition. Very recently, Behague, Chakraborti, and Liu~\cite{BehagueChakrabortiLiu} considered a different way of ``sabotaging'' a classical extremal theorem. They studied the maximum number of edges in a triangle-free graph that is required to contain a prescribed graph $\mathbb P$ as a subgraph; that is,
\[
\ex_{\mathbb P}(n,K_3)
=
\max\Bigl\{e(G): |V(G)|=n,\; G \text{ is } K_3\text{-free, and } \mathbb P\subseteq G\Bigr\}.
\]

In this paper we investigate analogous questions for $3$-uniform hypergraphs. Given a hypergraph property $\mathcal P$, let
\[
\ex_{\mathcal P}(n,\mathcal F)
\]
denote the maximum number of edges in an $\mathcal F$-free hypergraph on $n$ vertices satisfying the property $\mathcal P$.

Our focus is the Turán problem for the $3$-uniform hypergraph $F_5$, which consists of the three edges
\[
\{abc,\,abd,\,cde\}
\]
on the vertex set $\{a,b,c,d,e\}$. The classical Tur\'an problem for $F_5$ was solved by Frankl and F\"uredi \cite{FF}, who proved that for sufficiently large $n$ (in particular, for $n\ge 3000$),   the unique extremal hypergraph is the complete balanced tripartite $3$-uniform hypergraph. Consequently,
\[
\ex(n,F_5)=t_3^3(n),
\]
where $t_3^3(n)$ denotes the number of edges in the complete $3$-partite $3$-uniform hypergraph whose vertex classes differ in size by at most one.  Keevash and Mubayi \cite{KM} improved this result by only requiring $n\geq 33$.

Let $\Delta_2(H)$ denote the maximum codegree of a pair of vertices in a $3$-uniform hypergraph $H$. In this paper we study the Turán problem for $F_5$ under the additional condition that
\[
\Delta_2(H)\ge m,
\]
where $m= cn$ for some constant $\frac{1}{3}<c<1$. Equivalently, we seek the maximum number of edges in an $F_5$-free $3$-uniform hypergraph containing a pair of vertices with a large common neighborhood.

Our main result determines the corresponding extremal number exactly.

\begin{theorem}\label{thm:main}
For every constant $\frac{1}{3} <c<1$, there exists $n_0$ such that the following holds. If $n\ge n_0$ and $m=cn$, then
\[
\ex_{\Delta_2\ge m}(n,F_5)
=
m\left\lceil\frac{n-m}{2}\right\rceil
\left\lfloor\frac{n-m}{2}\right\rfloor.
\]
Moreover, the unique extremal hypergraph is
\[
K^3_{\,m,\,
\left\lceil \frac{n-m}{2}\right\rceil,\,
\left\lfloor \frac{n-m}{2}\right\rfloor}.
\]
\end{theorem}

We also prove a stability theorem of Theorem \ref{thm:main}. 

\begin{theorem}\label{thm:sta}
For every constant $\frac{1}{3} <c<1$ and every constant $C>0$, there exist constants $C'>0$ and $n_0$ such that the following holds. Let $n\ge n_0$ and $m=cn$, and let $H$ be an $F_5$-free $3$-uniform hypergraph on $n$ vertices with $\Delta_2(H)\ge m$. Suppose that $0\leq t \leq Cn^3$, if
\[
|E(H)|
\geq
m\left\lceil\frac{n-m}{2}\right\rceil
\left\lfloor\frac{n-m}{2}\right\rfloor
-t,
\]
then there exists a partition
\[
V(H)=V_1\cup V_2\cup V_3
\]
such that $|V_1| \geq m$ and
\[
|E(H)\,\Delta\,E(K^3(V_1,V_2,V_3))|
\le C't.
\]
If $t=o(n^2)$, then 
\[ H\subseteq K^3(V_1,V_2,V_3).\]
\end{theorem}

Our Theorem~\ref{thm:sta} is a hypergraph analogue of the following stability theorem of F\"{u}redi~\cite{F}.

\begin{quotedtheorem}[\cite{F}] \label{lem:furedi-sta}
Suppose that $G$ is $K_{r+1}$-free with $e(G) \geq t_r(n)-t$. Then there is a complete $r$-partite graph $K=K(V_1,\ldots,V_r)$ with $V(K)=V(G)$ such that
\[
|E(G)\Delta E(K)|\leq 3t.
\]
\end{quotedtheorem}

By comparison, Keevash and Mubayi \cite{KM} proved the following weaker stability result.
\begin{quotedtheorem}[\cite{KM}]
For any $\epsilon>0$, there exists $\delta>0$ such that if $H$ is an $F_5$-free $3$-uniform hypergraph on $n$ vertices with at least 
$(1-\delta)\lfloor \frac{n}{3}\rfloor\lfloor \frac{n+1}{3}\rfloor\lfloor \frac{n+2}{3}\rfloor$
edges, then there is a partition of the vertex
set of $H$ as $V=V_1\cup V_2 \cup V_3$ so that all but at most $\epsilon n^3$ edges of $H$ have one point
in each $V_i$.
\end{quotedtheorem}

More recently, Balogh, Clemen, and Luo~\cite{BaloghClemenLuo} studied the Turán number of $F^t_5$, where $F_5^t$ is the hypergraph on the vertex set $\{1,2,\ldots,5+t\}$ and edge set $\{123,124\} \cup \{34k: 5 \leq k \leq 5+t\}$ for any integer $t$. In particular, $F^0_5=F_5$. They showed that it has exponentially many extremal constructions. Very recently, the Turán number of the blow-up of $F_5$ on a different vertex has also been investigated; see \cite{ZCGK}. Let $F_2^t$ be the hypergraph on the vertex set $\{1,2,\ldots,3+t\}$ and edge set $ \{12k: 3 \leq k \leq 3+t\}$. Let
$\mathcal{H}(n,t)$ be the family of extremal $n$-vertex $F_2^t$-free $3$-graphs $H(n,t)$.

\begin{quotedtheorem}[\cite{BaloghClemenLuo}]
   For every integer $t\ge0$ and sufficiently large $n$,
   \[\ex(n, F^t_5)=\left\lfloor\frac{n}{3}\right\rfloor
\left\lfloor\frac{n+1}{3}\right\rfloor
\left\lfloor\frac{n+2}{3}\right\rfloor
+\operatorname{ex}\!\left(\left\lfloor\frac{n}{3}\right\rfloor,F_2^t\right)
+\operatorname{ex}\!\left(\left\lfloor\frac{n+1}{3}\right\rfloor,F_2^t\right)
+\operatorname{ex}\!\left(\left\lfloor\frac{n+2}{3}\right\rfloor,F_2^t\right).\]
Moreover, the extremal hypergraphs are $K^3_{\left\lfloor\frac{n}{3}\right\rfloor,\left\lfloor\frac{n+1}{3}\right\rfloor, \left\lfloor\frac{n+2}{3}\right\rfloor}$ with a copy of $H(|V_i|,t) \in \mathcal{H}(|V_i|,t)$ inside each part.
\end{quotedtheorem}

Our second main result is the Turán number of $F^t_5$ with $\Delta_2\ge m$.
\begin{theorem}\label{thm:main-blow-up}
For every fixed integer $t\ge0$ and every constant $\frac{1}{3} <c<1$, there exists $n_0$ such that the following holds. If $n\ge n_0$ and $m=cn$, then
\[
\ex_{\Delta_2\ge m}(n,F^t_5)
=
m\left\lceil\frac{n-m}{2}\right\rceil
\left\lfloor\frac{n-m}{2}\right\rfloor+\ex \left(m,F^t_2 \right)+\ex \left(\left\lceil\frac{n-m}{2}\right\rceil,F^t_2 \right)+\ex \left(\left\lfloor\frac{n-m}{2}\right\rfloor,F^t_2 \right).
\]
Moreover, the extremal hypergraphs are $K^3_{\,m,\,
\left\lceil \frac{n-m}{2}\right\rceil,\,
\left\lfloor \frac{n-m}{2}\right\rfloor}$ with a copy of $H(|V_i|,t) \in \mathcal{H}(|V_i|,t)$ inside each part.
\end{theorem}

For fixed $t\ge0$, let
\[
\Phi_t(n,m)
=m\left\lceil\frac{n-m}{2}\right\rceil
\left\lfloor\frac{n-m}{2}\right\rfloor+\ex(m,F_2^t)
+\ex\!\left(\left\lceil\frac{n-m}{2}\right\rceil,F_2^t\right)
+\ex\!\left(\left\lfloor\frac{n-m}{2}\right\rfloor,F_2^t\right),
\]
denote the expression in Theorem \ref{thm:main-blow-up}.

For a partition $V=V_1\cup V_2\cup V_3$, let
$\mathcal T_t(V_1,V_2,V_3)$ denote the family of all hypergraphs of the form
\[
K^3(V_1,V_2,V_3)\cup J_1\cup J_2\cup J_3,
\]
where each $J_i$ is an extremal $F_2^t$-free $3$-graph on $V_i$.
The following stability theorem is the corresponding near-extremal version of
Theorem~\ref{thm:main-blow-up}.

\begin{theorem}[Stability for $F_5^t$]\label{thm:sta-blow-up}
For every fixed integer $t\ge0$, every constant $\frac{1}{3} <c<1$, and every
constant $C>0$, there exist constants $C'>0$ and $n_0$ such that the following
holds. Let $n\ge n_0$ and $m=cn$, and let $H$ be an $F_5^t$-free $3$-uniform
hypergraph on $n$ vertices with $\Delta_2(H)\ge m$. Suppose that
$0\le \tau\le Cn^3$ and
\[
|E(H)|\ge \Phi_t(n,m)-\tau.
\]
Then there exist a partition
\[
V(H)=V_1\cup V_2\cup V_3
\]
and a hypergraph $G\in\mathcal T_t(V_1,V_2,V_3)$ such that
\[
|E(H)\,\Delta\,E(G)|\le C'\tau.
\]
%In particular, if $\tau=0$, then $H$ is one of the extremal hypergraphs from Theorem~\ref{thm:main-blow-up}.
\end{theorem}

By comparison, Balogh, Clemen, and Luo \cite{BaloghClemenLuo} proved the following tability result for $F^t_5$.
\begin{quotedtheorem}[\cite{BaloghClemenLuo}]
For every $t \geq 0$ and $\varepsilon>0$, there exists $\delta>0$ and $n_0>0$ such that for every $n \geq n_0$ the following statement holds. Let $H$ be an $F^t_5$-free $3$-uniform hypergraph on $n$ vertices with at least 
$(\frac{1}{27}-\delta)n^3$ edges, then there is a partition of the vertex
set of $H$ as $V=V_1\cup V_2 \cup V_3$ so that all but at most $\varepsilon n^3$ edges of $H$ have one point in each $V_i$.
\end{quotedtheorem}

The rest of the paper is organized as follows. Section 2 contains the notation while Section 3 contains the proofs of Theorems~\ref{thm:main} and~\ref{thm:sta}. Section 4 proves Theorems~\ref{thm:main-blow-up} and~\ref{thm:sta-blow-up}.

\section{Notation and Lemmas}

A \emph{$3$-uniform hypergraph} (or \emph{$3$-graph}) $H=(V,E)$ consists of a vertex set $V$ and an edge set $E\subseteq \binom{V}{3}$. Throughout the paper, all hypergraphs are assumed to be $3$-uniform unless stated otherwise. We write $V(H)$ and $E(H)$ for the vertex set and edge set of $H$, respectively, and denote by $v(H)=|V(H)|$ and $e(H)=|E(H)|$ the numbers of vertices and edges.

A hypergraph $H'$ is a \emph{subhypergraph} of $H$ if
\[
V(H')\subseteq V(H)
\quad \text{and} \quad
E(H')\subseteq E(H).
\]

We say that $H$ is \emph{$F$-free} if it contains no subhypergraph isomorphic to $F$.

For a vertex $v\in V(H)$, the \emph{neighborhood} of $v$ is the set
\[
N_H(v)=\big\{\{x,y\}\in \binom{V(H)\setminus\{v\}}{2}: \{v,x,y\}\in E(H)\big\}.
\]
The \emph{degree} of $v$ is
\[
d_H(v)=|N_H(v)|.
\]

For a vertex $v\in V(H)$, the \emph{link graph} of $v$, denoted by $L_H(v)$, is the graph with vertex set
\[
V(L_H(v))=V(H)\setminus\{v\}
\]
and edge set
\[
E(L_H(v))
=
\Big\{
xy\in \binom{V(H)\setminus\{v\}}{2}
:
vxy\in E(H)
\Big\}.
\]
Thus the edges of $L_H(v)$ correspond precisely to the hyperedges of $H$ containing $v$. In particular,
\[
e(L_H(v))=d_H(v).
\]

When the ambient hypergraph is clear from the context, we omit the subscript $H$ and write $N(v)$, $d(v)$, and $L(v)$.

For a pair of distinct vertices $u,v\in V(H)$, the \emph{neighborhood} of $\{u,v\}$ is
\[
N_H(u,v)=\{w\in V(H): uvw\in E(H)\}.
\]
The \emph{codegree} of the pair $\{u,v\}$ is
\[
d_H(u,v)=|N_H(u,v)|.
\]
The \emph{maximum codegree} of $H$ is defined by
\[
\Delta_2(H)=\max_{\{u,v\}\in \binom{V(H)}{2}} d_H(u,v).
\]

The \emph{shadow graph} (or \emph{$2$-shadow}) of $H$, denoted by $\partial H$, is the graph on vertex set $V(H)$ whose edge set consists of all pairs of vertices that are contained in some hyperedge of $H$, namely
\[
E(\partial H)
=
\Big\{
uv\in \binom{V(H)}{2}:
\exists\, w\in V(H)\text{ such that } uvw\in E(H)
\Big\}.
\]

A subset $S\subseteq V(H)$ is called \emph{strongly independent} in $H$ if it is an independent set in the shadow graph $\partial H$. Equivalently, no two vertices of $S$ are contained in a common hyperedge of $H$. That is,
\[
\binom{S}{2}\cap E(\partial H)=\emptyset.
\]

In other words, every hyperedge of $H$ contains at most one vertex from $S$.
For pairwise disjoint sets $V_1,V_2,V_3$, we write
\[
K^3(V_1,V_2,V_3)
\]
for the complete tripartite $3$-uniform hypergraph with parts $V_1,V_2,V_3$, whose edge set consists of all triples having exactly one vertex in each part. When $|V_i|=n_i$ for $i\in[3]$, we also write
\[
K^3_{n_1,n_2,n_3}.
\]

\section{Proof of Theorem \ref{thm:sta} and Theorem~\ref{thm:main}}
Let $r=\Delta_2(H)\ge m$. For a fixed $r$, it is easy to see that the product $|V_1||V_2||V_3|$ is maximized when the two remaining parts are as equal as possible. Now we will show that the function $f(n,r)$ is nonincreasing for $r\ge \lceil n/3\rceil$. 

Indeed, if $n-r=2a$, then
\[
f(n,r)-f(n,r+1)=ra^2-(r+1)a(a-1)=a(r+1-a)\ge0,
\]
and if $n-r=2a+1$, then
\[
f(n,r)-f(n,r+1)=ra(a+1)-(r+1)a^2=a(r-a)\ge0.
\]

A $k$-core of a hypergraph $H$ is a maximal subhypergraph of $H$ in which every pair $uv$ satisfies either $d_{uv}\ge k$ or $d_{uv}=0$.

Define
\[
f(n,m)=m\left\lceil \frac{n-m}{2}\right\rceil \left\lfloor \frac{n-m}{2}\right\rfloor.
\]

\begin{lemma}\label{lem:k4free}
Suppose that $H$ is an $F_5$-free $3$-uniform hypergraph on $n$ vertices containing a strongly independent set $M$ of size $m$, where $m\ge cn$ and $\frac{1}{3} <c < 1$ is fixed. Assume further that the shadow graph $\partial(H)$ is $K_4$-free. Then
\[
|E(H)|\le f(n,m).
\]

If 
\[
|E(H)|\ge f(n,m)-t,
\]
then there exists a partition 
\[
V(H)=V_1\cup V_2\cup V_3
\]
such that $V_1=M$ and
\[
|E(H) \, \Delta \, E(K^3(V_1,V_2,V_3))|
\le O(t).
\]
\end{lemma}

\begin{proof}
Since $M$ is strongly independent, every edge of $H$ contains at most one vertex of $M$. Hence every edge either contains exactly one vertex of $M$ or is entirely contained in $\overline M:=V(H)\setminus M$.

Choose a vertex $v_0\in M$ with maximum degree among all vertices of $M$. Let
\[
F=E(L_H(v_0))
\subseteq
\binom{\overline M}{2},
\]
and let $S\subseteq \overline M$ denote the set of vertices incident with at least one edge of $F$. Write $s=|S|$.

Since $L_H(v_0)$ is triangle-free, Mantel's theorem yields
\[
|F|
\le
\left\lfloor\frac{s}{2}\right\rfloor
\left\lceil\frac{s}{2}\right\rceil
\le
\frac{s^2}{4}.
\]

\begin{claim}\label{clm:no-edge-in-S}
The induced subhypergraph $H[S]$ contains no edge.
\end{claim}

\begin{proof}
Suppose that $abc\in E(H[S])$. Since $a,b,c\in S$, each of $a,b,c$ is adjacent to $v_0$ in the shadow graph $\partial(H)$. Furthermore, the pairs
\[
ab,\quad ac,\quad bc
\]
belong to $E(\partial(H))$ because $abc$ is a hyperedge.

Therefore
\[
\partial(H)[\{v_0,a,b,c\}]
\cong K_4,
\]
contradicting the assumption that $\partial(H)$ is $K_4$-free.
\end{proof}

We classify the edges of $H$ into four types.

\medskip

\noindent
{\bf Type I.}
Edges containing exactly one vertex of $M$.

Since $v_0$ has maximum degree among vertices of $M$,
\[
d(v)\le d(v_0)=|F|
\qquad
\text{for all }v\in M.
\]
Hence the number of Type I edges is at most
\[
\sum_{v\in M} d(v)
\le m|F|.
\]

\medskip

\noindent
{\bf Type II.}
Edges contained in $\overline M$ having two vertices in $S$ and one vertex in $\overline M\setminus S$.

Let $T$ be the set of pairs $yz\in \binom{S}{2}$ for which there exists a vertex $x\in S$ such that
\[
xy,xz\in F.
\]
By Claim~\ref{clm:no-edge-in-S}, no pair in $T$ can appear in an edge of Type II.

Let $a_v$ denote the degree of $v$ in the link graph $L_H(v_0)$. Then
\[
\sum_{v\in S} a_v =2|F|.
\]

For a pair $ij\in \binom{S}{2}$, let $b_{ij}$ denote the number of common neighbors of $i$ and $j$ in $L_H(v_0)$.
We have
\begin{align}
    \sum_{ij\in \binom{S}{2}} b_{ij} &=\sum_{v\in S} \binom{a_v}{2}. \label{eq:bij}\\
    \sum_{ij\in T} \binom{b_{ij}}{2} &= 2\#C_4. \label{eq:bij2}
\end{align}
Here $\#C_4$ denotes the number of induced copies of $C_4$ in $L_H(v_0)$. Since $L_H(v_0)$ is $K_3$-free, we have
\begin{equation}
    4\#C_4 \leq \sum_{uv\in F}(a_u-1)(a_v-1)
    \leq  \sum_{uv\in F}\frac{(s-2)^2}{4}
    \leq |F|\frac{(s-2)^2}{4}.
    \label{eq:C4}
\end{equation}
Applying the Cauchy--Schwarz inequality to \eqref{eq:bij},
we get
\[\sum_{ij\in \binom{S}{2}} b_{ij} \geq
     \frac{2|F|^2}{s} -|F|.
\]
We have
\begin{align*}
    |T| &\geq \frac{\left(\sum_{ij\in T} b_{ij}\right)^2}{\sum_{ij\in T} b_{ij}^2}\\
    &= \frac{\left(\sum_{ij\in T} b_{ij}\right)^2}{4\#C_4+ \sum_{ij\in T}b_{ij}}\\
    &\geq \frac{\left(\frac{2|F|^2}{s} -|F| \right)^2}{|F|(s-2)^2/4+ \frac{2|F|^2}{s} -|F|}\\
    &= \frac{\left(\frac{2|F|}{s} -1 \right)^2|F|}{(s-2)^2/4+ \frac{2|F|}{s} -1}.\\
\end{align*}
\medskip

\noindent
{\bf Type III.}
Edges contained in $\overline M$ having exactly one vertex in $S$.

Applying the trivial bound, the number of such edges is at most
\[
s \frac{(n-m-s)^2}{4}.
\]

\medskip

\noindent
{\bf Type IV.}
Edges contained in $\overline M\setminus S$.

Since $\partial(H[\overline M\setminus S])$ is $K_4$-free, every hyperedge of
$H[\overline M\setminus S]$ corresponds to a triangle in the shadow graph.
By Turán's theorem, the number of triangles in a $K_4$-free graph on
$n-m-s$ vertices is at most
\[
\left(\frac{n-m-s}{3}\right)^3.
\]
Hence the number of Type IV edges is at most
\[
\left(\frac{n-m-s}{3}\right)^3.
\]

Combining the four estimates gives
\begin{align*}
|E(H)|
&\le
m|F|
+
\left(
\binom{s}{2}
-
\frac{\left(\frac{2|F|}{s} -1 \right)^2|F|}{(s-2)^2/4+ \frac{2|F|}{s} -1}
\right)(n-m-s)\\
&\qquad
+
s\frac{(n-m-s)^2}{4}
+
\left(\frac{n-m-s}{3}\right)^3.
\end{align*}

Let $|F|=x$ and define
\[
g(x,s)
=mx
-
\frac{\left(\frac{2x}{s} -1 \right)^2x}{(s-2)^2/4+ \frac{2x}{s} -1}
(n-m-s)
+
\frac{s(s-1)(n-m-s)}{2}
+
\frac{s(n-m-s)^2}{4}
+
\left(\frac{n-m-s}{3}\right)^3.
\]

We first claim that
\[
s>\frac{26}{31}(n-m).
\]

Suppose instead that
\[
s\le \frac{26}{31}(n-m).
\]
Since every edge containing a vertex of $M$ contributes to the degree of a vertex in $M$,
and the number of edges in $H[\overline M]$ is at most $\left(\frac{n-m}{3}\right)^3$,
\[
|E(H)|
\le
m|F|
+
\left(\frac{n-m}{3}\right)^3.
\]
Using $|F|\le s^2/4$, we obtain
\begin{align*}
|E(H)|
&\le
m\frac{s^2}{4}
+
\left(\frac{n-m}{3}\right)^3\\
&\le
\frac{13^2}{31^2}m(n-m)^2
+
\frac{2}{27}m(n-m)^2\\
&<
0.24994\,m(n-m)^2\\
&<
f(n,m)-Cn^2
\end{align*}
for sufficiently large $n$, a contradiction.

Hence
\[
s>\frac{26}{31}(n-m).
\]

Since $s\geq \frac{26}{31}(n-m)$, we have
$n-m-s\leq \frac{5}{31}(n-m)$ and
$m\geq \frac{n-m}{2}>3.1 (n-m-s)$.

We show that 
\[g(x,s)\leq g(\frac{s^2}{4},s)  \qquad \mbox{ for } x\leq \frac{s^2}{4},\]

Indeed,
\begin{align*}
g\left(\frac{s^2}{4},s\right)-g(x,s)
&=m\left(\frac{s^2}{4}-x\right)
-(n-m-s)\left(Q\left(\frac{s^2}{4}\right)-Q(x)\right),
\end{align*}
where
\[
Q(x)=
\frac{\left(\frac{2x}{s} -1 \right)^2x}{(s-2)^2/4+ \frac{2x}{s} -1}.
\]
Note $\frac{m}{n-m-s}>3.1$, $0\leq x \leq \frac{s^2}{4}$.

Set $b=\frac{s^2}{4}$.
For \(x<b\), set
\[
u=\frac{2x}{s}.
\]
Then \(0\le u\le s/2\). A direct simplification gives
\[
\frac{Q(b)-Q(x)}{b-x}
=
R(u),
\]
where
\[
R(u)
=
\frac{s^2+2su-6s+4u^2-8u+8}{s^2-4s+4u}.
\]

We claim that \(R(u)\) is increasing on \(0\le u\le s/2\). Indeed,
\[
R'(u)
=
\frac{
2\left(
s^3+4s^2u-10s^2-16su+28s+8u^2-16
\right)
}
{(s^2-4s+4u)^2}.
\]
Since
\begin{align*}
s^3+4s^2u-10s^2-16su+28s+8u^2-16
&=(s^3-10s^2+28s-16)+(4s^2u-16su+8u^2),
\end{align*}
and
\[
s^3-10s^2+28s-16=(s-4)(s^2-6s+4)>0,
\]
while
\[
4s^2u-16su+8u^2=4su(s-4)+8u^2\ge 0.
\]
Hence \(R'(u)>0\) on \(0\le u\le s/2\). Therefore
\[
R(u)\le R(s/2).
\]
A direct computation gives
\[
R(s/2)=3-\frac4s.
\]
Thus
\[
\frac{Q(b)-Q(x)}{b-x}
\le
3-\frac4s,
\]
so
\[
Q(b)-Q(x)
\le
\left(3-\frac4s\right)(b-x).
\]

Therefore
\begin{equation} \label{eq:gap}
g(b,s)-g(x,s)
\ge
m(b-x)-(n-m-s)\left(3-\frac4s\right)(b-x)
\ge (n-m-s) \left(0.1 +\frac4s\right)(b-x)\ge 0.
\end{equation}

If $s=n-m$,
then
\[
|E(H)|\leq m|F| \leq
g\!\left(
\left\lceil\frac{n-m}{2}\right\rceil
\left\lfloor\frac{n-m}{2}\right\rfloor,
n-m
\right)
=
f(n,m).
\]

If $s<n-m$, then
\begin{align*}
|E(H)|
&\le
g\!\left(\frac{s^2}{4},s\right)
\\
&=
\frac{s^2}{4}m
+
(n-m-s)
\Bigg(
\frac{s(n-m)}{4}
+\frac{(n-m-s)^2}{27} 
\Bigg)\\
&\leq \frac{(n-m)^2m}{4} -2m\\
&<f(n,m).
\end{align*}

Suppose that
\[
|E(H)|>f(n,m)-t.
\]
We have
\[\frac{s^2}{4}m
+
(n-m-s)
\Bigg(
\frac{s(n-m)}{4}
+\frac{(n-m-s)^2}{27} 
\Bigg) \geq \frac{(n-m)^2m}{4} -t. \]

Since $\frac{m}{n-m}\ge\frac{c}{1-c}>\frac{1}{2}$, we have
\begin{align*}
    t &\geq \frac{m(n-m-s)(n-m+s)}{4}-(n-m-s)\Bigg(
\frac{s(n-m)}{4}
+\frac{(n-m-s)^2}{27} 
\Bigg)\\
&\ge\Big(\frac{c}{1-c}-\frac{1}{2} \Big)\frac{(n-m)(n-m-s)(n-m+s)}{4}-\frac{(n-m)(n-m-s)(n-m+s)}{8}\\
&\quad -(n-m-s)\Bigg(
\frac{s(n-m)}{4}
+\frac{(n-m-s)^2}{27} 
\Bigg)\\
&=\Big(\frac{c}{1-c}-\frac{1}{2} \Big)\frac{(n-m)(n-m-s)(n-m+s)}{4}+(n-m-s)\Bigg(\frac{(n-m-s)(n-m)}{8}-\frac{(n-m-s)^2}{27} \Bigg)\\
& \geq \Big(\frac{c}{1-c}-\frac{1}{2} \Big)\frac{(n-m)(n-m-s)(n-m+s)}{4}.
\end{align*}

Thus, we get
\begin{equation} \label{eq:s}
n-m-s \leq \frac{8(1-c)t}{(3c-1)(n-m)(n+m+s)}= O\Big(\frac{t}{n^2}\Big).
\end{equation}

In particular, if $t<Cn^2$, then $n-m-s=O(1)$. Furthermore,
if $t=o(n^2)$, then $n-m-s=0$. 

Observe that $g(|F|,s)\geq g(s^2/4,s)-t$. By \eqref{eq:gap}, we have
\[t \geq m\Big(\frac{s^2}{4}-|F| \Big)- (n-m-s)\Big(3-\frac{4}{s} \Big) \Big(\frac{s^2}{4}-|F| \Big).\]
Thus, by \eqref{eq:s}, we have
\[\frac{s^2}{4}-|F|\leq \frac{t}{m - (n-m-s)(3-\frac{4}{s})} \leq O\Big(\frac{t}{n}\Big).\]

By Lemma~\ref{lem:furedi-sta}, there exists a partition
\[
S=V'_2\cup V'_3
\]
such that
\begin{equation} \label{eq:furedi}
|E(F)\Delta E(K(V'_2,V'_3))|\leq  O\Big(\frac{t}{n}\Big).    
\end{equation}

Let $|V'_2|=\frac{s}{2}+r$ and $|V'_3|=\frac{s}{2}-r$. Then 
\[\frac{s^2}{4}-r^2=|E(K(V'_2,V'_3))| \geq |F|- O\Big(\frac{t}{n}\Big) \geq \frac{s^2}{4}-O\Big(\frac{t}{n}\Big).\]
Thus, $r \leq O(\sqrt{t/n}) <O(n)$ and so $|V'_2|, |V'_3| =\Theta(n)$.

Let $F_x=L_H(x)$ for $x \in M$. In particular, $F_{v_0}=F$.
\begin{claim} \label{cla:edge-inside}
\[
\left|E\Big(\bigcup_{x \in M} (F_x[V'_2] \cup F_x[V'_3])\Big)\right|\leq  O\Big(\frac{t}{n}\Big).
\]  
\end{claim}

We prove the claim. Observe $E(F[V'_2] \cup F[V'_3]) \leq O(t/n)$ by \eqref{eq:furedi}. Without loss of generality, consider an edge $uv \in E(F_x[V'_2])$ for some $x \in M \setminus \{v_0\}$. If there is a vertex $w \in V'_3$ such that $vw, uw \in E(F)$, then $uwv_0, vwv_0, uvx$ form a copy of $F_5$, a contradiction. Thus, at least one of $u,v$ has at least $|V'_3|/2=\Theta(n)$ missing crossing edges between $V'_2$ and $V'_3$. By \eqref{eq:furedi}, at most $O(t/n^2)$ vertices in $S$ can have at least $\Theta(n)$ missing crossing edges between $V'_2$ and $V'_3$. Each edge inside $F_x[V'_2]$ or $F_x[V'_3]$ must be adjacent to such a vertex. Hence, the statement follows.

Extend this partition to $\overline{M}=V_2\cup V_3$ by assigning the remaining vertices of $\overline{M}\setminus S$ arbitrarily to $V_2$ and $V_3$. Let $V_1=M$. Consider the partition $V=V_1\cup V_2\cup V_3$. A triple $u,v,w$ is called {\em bad} if either
\begin{enumerate}
    \item $e=uvw \in E(K^3(V_1,V_2,V_3)) \setminus E(H)$, or 
    \item  $e=uvw \in E(H) \setminus E(K^3(V_1,V_2,V_3))$.
\end{enumerate}

Since $|\overline{M} \setminus S| \leq  O(t/n^2)$, the number of bad triples incident to $\overline{M} \setminus S$ is at most $O(t)$. Now we look at the bad triples that are only incident to $V_1 \cup V'_2 \cup V'_3$. By Claim~\ref{cla:edge-inside}, there are at most $O(t)$ bad triples $u,v,w$ such that $e=uvw \in E(H[V_1 \cup V'_2 \cup V'_3]) \setminus E(K^3(V_1,V_2,V_3))$. Let $X$ be the set of bad triples $u,v,w$ such that $e=uvw \in E(K^3(V_1,V_2,V_3)) \setminus E(H[V_1 \cup V'_2 \cup V'_3])$. We observe that \[f(n,m)-t < |E(H)| \leq f(n,m)-|X|+O(t).\]
Thus, $|X| \leq O(t)$.
It follows that
\[
|E(H)\Delta E(K^3(V_1,V_2,V_3))|
=O(t).
\]

This completes the proof of the lemma.
\end{proof}

\begin{lemma}\label{lem:3-core-sta}
For every constant $\frac{1}{3}<c<1$ and every constant $C>0$, there exist constants $C'>0$ and $n_0$ such that the following holds. Let $n\ge n_0$, let $m>cn$, and let $H$ be an $F_5$-free $3$-uniform hypergraph on $n$ vertices with $\Delta_2(H)\ge m$ such that either $d_H(u,v) \geq 3$ or $d_H(u,v)=0$ for any pair of vertices $u,v \in V(H)$. For any $t=o(n^3)$, if
\[
|E(H)|
>
m\left\lceil\frac{n-m}{2}\right\rceil
\left\lfloor\frac{n-m}{2}\right\rfloor
-t,
\]
then there exists a partition
\[
V(H)=V_1\cup V_2\cup V_3
\]
such that $|V_1|=\Delta_2(H)\geq m$ and
\[
|E(H)\,\Delta\,E(K^3(V_1,V_2,V_3))|
\le O(t).
\]
\end{lemma}

\begin{proof}
Choose a pair of vertices $a,b$ with $\Delta_2(G)=d_H(a,b)=r\ge m$, and let $M= N_H(a,b)$. Thus, $|M|=r \geq m$.
\begin{claim} \label{cla:indep}
The set $M$ is strongly independent in $H$.
\end{claim}

\begin{proof}
Suppose otherwise. Then there exist vertices $c,d\in M$ such that $cd\in E(\partial(H))$. Thus, we have $d_{H}(c,d)\ge 3$. Therefore, there exists an edge $e=cdf \in E(H)$ with $f\notin\{a,b\}$. Since $c,d \in M$, we have $abc,abd \in E(H)$. Then $abc,abd,cdf$ forms a copy of $F_5$, a contradiction.
\end{proof}

If $\partial H$ is $K_4$-free, then by Lemma~\ref{lem:k4free}, we are done. Thus we may assume that $\partial(H)$ contains a copy of $K_4$. Let its vertex set be $\{w,x,y,z\}$.

\begin{claim}\label{clm:K4}
Let $wxyz$ be a copy of $K_4$ in $\partial(H)$. Then for every pair
\[
uv\in \binom{V(H)}{2}
\setminus
\{wx,wy,wz,xy,xz,yz\},
\]
we have
\[
|N_{H}(uv)\cap\{w,x,y,z\}|
\le 1.
\]
\end{claim}

\begin{proof}
Suppose not. Without loss of generality, assume $uvx, uvy\in E(H)$. Since $xy\in E(\partial(H))$ and $H$ is a $3$-core, we have $d_H(x,y)\ge 3$. Thus, there exists a vertex $p\notin\{u,v\}$ such that $xyp\in E(H)$. Then $uvx, uvy, xyp$ forms a copy of $F_5$, a contradiction.
\end{proof}

We now iteratively eliminate all copies of $K_4$ from the shadow graph. Whenever a copy of $K_4$ with vertex set $\{w,x,y,z\}$ appears, we delete all four vertices. Repeating this procedure yields a sequence of induced subhypergraphs
\[
H=H_0\supset H_1\supset\cdots\supset H_q,
\]
such that $\partial(H_q)$ is $K_4$-free.

Since $M$ is strongly independent, every copy of $K_4$ in $\partial(H_i)$ contains at most one vertex of $M$. Let
\[
q=q_1+q_2,
\]
where
\begin{itemize}
\item $q_1$ is the number of deleted copies of $K_4$ containing a vertex of $M$;
\item $q_2$ is the number of deleted copies of $K_4$ disjoint from $M$.
\end{itemize}

Without loss of generality, assume that all $q_1$ copies are removed before the $q_2$ copies.

Applying Lemma~\ref{lem:k4free} to $H_q$, we obtain
\[
|E(H_q)|
\le
f(n-4q,r-q_1).
\]

For a deleted copy of $K_4$ containing one vertex of $M$, Claim~\ref{clm:K4} implies that the number of incident edges is at most
\[
\binom{n-4i}{2}
-
\binom{r-i}{2}
+
6(n-4i-2).
\]

Similarly, for a deleted copy of $K_4$ disjoint from $M$, the number of incident edges is at most
\[
\binom{n-4q_1-4i}{2}
-
\binom{r-q_1}{2}
+
6(n-q_1-4i-2).
\]

Therefore,
\begin{align}
|E(H)|
&\le
f(n-4q,r-q_1)
\nonumber\\
&\quad+
\sum_{i=1}^{q_1}
\left(
\binom{n-4i}{2}
-
\binom{r-i}{2}
+
6(n-4i-2)
\right)
\nonumber\\
&\quad+
\sum_{i=1}^{q_2}
\left(
\binom{n-4q_1-4i}{2}
-
\binom{r-q_1}{2}
+
6(n-q_1-4i-2)
\right).
\label{eq:qbound}
\end{align}

Since 
\begin{equation} \label{eq:Elb}
|E(H)| \geq f(n,m)-t \geq f(n,r)-t,    
\end{equation}

combining \eqref{eq:qbound} with \eqref{eq:Elb} and simplifying, we obtain

\[q \leq \frac{8t}{(1-c)(3c-1)n^2}.\]

Since $c>\frac13$, we have
\[q=O\Big(\frac{t}{n^2}\Big).\]

Since $H_q$ is $K_4$-free and \[|E(H_q)| \geq |E(H)|-qn^2>f(n-4q,r-q_1)-O(t),\]
applying Lemma~\ref{lem:k4free} to $H_q$, we get a partition of $V(H_q)=V'_1 \cup V'_2 \cup V'_3$ with $|V'_1|=r-q_1$ such that \[|E(H_q) \Delta E(K^3(V'_1,V'_2,V'_3))|\le O(t).\]
Now adding back $q_1$ vertices to $V'_1$ and adding back the remaining vertices arbitrarily. Since $4q=O(t/n^2)$ vertices were removed, we have the symmetric difference by at most $4qn^2 \leq O(t)$. Thus, \[|E(H) \Delta E(K^3(V_1,V_2,V_3))|\le O(t).\]
\end{proof}

Let $|V_1|=r$, $|V_2|=a_2$ and $|V_3|=a_3$ where $a_2+a_3=n-r$. For a given $r$, it is easy to see that the number of edges is maximized when the two parts $V_2, V_3$ are as equally as possible, i.e.,
\[
|E(H)|
\le
r\left\lceil\frac{n-r}{2}\right\rceil
 \left\lfloor\frac{n-r}{2}\right\rfloor
=f(n,r).
\]
Now we will show that the function $f(n,r)$ is nonincreasing for $r\ge \lceil n/3\rceil$. Indeed, if $n-r=2a$, then
\[
f(n,r)-f(n,r+1)=ra^2-(r+1)a(a-1)=a(r+1-a)\ge0,
\]
and if $n-r=2a+1$, then
\[
f(n,r)-f(n,r+1)=ra(a+1)-(r+1)a^2=a(r-a)\ge0.
\]
Hence $f(n,r+1)\le f(n,r)$.

\begin{proof}[Proof of Theorem~\ref{thm:main}]
Suppose $|E(H)|\geq f(n,m)$ otherwise there is nothing to prove. Choose a pair of vertices $a,b$ with $\Delta_2(G)=d_H(a,b)=r\ge m$, and let $M= N_H(a,b)$. Consider the $3$-core $H'$ of $H$, obtained by repeatedly deleting every edge containing a pair $uv$ with codegree $d_H(u,v)\in\{1,2\}$. Since at most two edges are removed for each pair of vertices, we have
\begin{equation}
|E(H')|
\geq
|E(H)|-2\binom{n}{2}
\geq f(n,m)-O(n^2).
\end{equation}

Applying Lemma~\ref{lem:3-core-sta} to $H'$, we have a partition with $|V_1|=r$ such that
\begin{equation} \label{eq:Hprime-n2}
|E(H') \Delta E(K^3(V_1,V_2,V_3))|\le O(n^2).    
\end{equation}

First, we show that 
\begin{equation} \label{eq:bt-n2}
 |E(H')| \geq |E(H)|-O(n) \geq f(n,m)-O(n) 
\end{equation}

 Let $i,j,k \in [3]$. Define the following sets
\[X_1=\{(u,v): u \in V_i, v\in V_j \},\]
\[X_2=\{(u,v):u,v \in V_i \}.\]

We say a triple $u,v,w$ is {\em missing} if $e=uvw \in E(K^3(V_1,V_2,V_3)) \setminus E(H')$. By \eqref{eq:Hprime-n2}, $H'$ has at most $O(n^2)$ missing triples. Let $Y$ be the set of pairs $(u,v)$ with $d_H(u,v) \in \{1,2\}$. If $(u,v) \in Y \cap X_1$, then $(u,v)$ is contained in at least $O(n)$ missing triples and thus $|Y \cap X_1| \leq O(n)$.

Let $(u,v) \in Y \cap X_2$ and $w \in N_H(u,v)$. Without loss of generality, assume $w \in V_j$. We observe that for any pair of vertices $x,y$ where $x \in V_j \setminus \{w\}$ and $y \in V_k$, if $xyu \in E(H')$, then $xyv \notin E(H')$. If not, then $xyu, xyv, uvw$ forms a copy of $F_5$. A contradiction. Thus, at least one of $u,v$ is contained in at least $(|V_j|-1)|V_k|/2 \geq O(n^2)$ missing triples. We say a vertex is {\em bad} if it is contained in at least $O(n^2)$ missing triples. Since $H'$ has at most $O(n^2)$ missing triples, the number of bad vertices is at most $O(1)$. Since any pair of vertices in $Y \cap X_2$ must contain a bad vertex, we have $|Y \cap X_2| \leq O(n)$. Hence, \eqref{eq:bt-n2} follows. Applying Lemma~\ref{lem:3-core-sta} to $H'$, we have a partition with $|V_1|=r$ such that
\begin{equation} \label{eq:Hprime-n}
|E(H') \Delta E(K^3(V_1,V_2,V_3))|\le O(n).    
\end{equation}

We now show that 
\begin{equation} \label{eq:bt-n}
 |E(H')| \geq |E(H)|-O(1) \geq f(n,m)-O(1).   
\end{equation}

By \eqref{eq:Hprime-n}, $H'$ has at most $O(n)$ missing triples. If $(u,v) \in Y \cap X_1$, then $(u,v)$ is contained in at least $O(n)$ missing triples and thus $|Y \cap X_1| \leq O(1)$. Since a bad vertex is contained in at least $O(n^2)$ missing triples and any pair of vertices in $Y \cap X_2$ contains a bad vertex, we have $|Y \cap X_2|=0$. Hence, \eqref{eq:bt-n} follows. Applying Lemma~\ref{lem:3-core-sta} to $H'$, we have a partition with $|V_1|=r$ such that
\[|E(H') \Delta E(K^3(V_1,V_2,V_3))|\le O(1).\]

Then $H'$ has at most $O(1)$ missing triples. If $(u,v)\in Y \cap X_1$, then $(u,v)$ is contained in at least $O(n)$ missing triples and thus $|Y \cap X_1|=0$. Also, it is easy to see $|Y \cap X_2|=0$. Thus, we have $H=H' \subseteq K^3(V_1,V_2,V_3)$. Hence $|E(H)| \leq |E(K^3(V_1,V_2,V_3))| \le f(n,r)\le f(n,m)$, which proves the upper bound. If equality holds, then $r=m$ and the two remaining parts must be as equal as possible. Thus the unique extremal hypergraph is
\[
K^3_{m,\lceil (n-m)/2\rceil,\lfloor (n-m)/2\rfloor}.
\]
\end{proof}

\begin{proof}[Proof of Theorem~\ref{thm:sta}]

Choose a pair of vertices $a,b$ with $\Delta_2(G)=d_H(a,b)=r\ge m$, and let $M= N_H(a,b)$. Consider the $3$-core $H'$ of $H$, obtained by repeatedly deleting every edge containing a pair $uv$ with codegree $d_H(u,v)\in\{1,2\}$. Since at most two edges are removed for each pair of vertices, we have
\begin{equation}
|E(H')|
\ge
|E(H)|-2\binom{n}{2}
\geq f(n,m)-t-O(n^2).
\end{equation}

If $t$ is at least a sufficiently large constant multiple of $n^2$, then the desired bound follows by first applying Lemma~\ref{lem:3-core-sta} to $H'$ and then adding back the deleted edges which are at most $O(n^2)=O(t)$. Thus, we may assume $t=o(n^2)$. Applying Lemma~\ref{lem:3-core-sta} to $H'$, we have a partition with $|V_1|=r$ such that
\begin{equation} \label{eq:Hprime-n2}
|E(H') \Delta E(K^3(V_1,V_2,V_3))|\le O(n^2).    
\end{equation}

Let $Y$ be the set of pairs $(u,v)$ with $d_H(u,v) \in \{1,2\}$. Using the same argument as in the proof of Theorem \ref{thm:main}, we obtain $|Y|=|Y \cap X_1|+|Y \cap X_2| \leq O(n)$ and thus
\begin{equation} \label{eq:bt-n2}
 |E(H')| \geq |E(H)|-O(n) \geq f(n,m)-t-O(n).
\end{equation}

If $t$ is at least a sufficiently large constant multiple of $n$, then the desired bound follows by first applying Lemma~\ref{lem:3-core-sta} to $H'$ and then adding back the deleted edges which are at most $O(n)=O(t)$. Thus, we may assume $t=o(n)$. Applying Lemma~\ref{lem:3-core-sta} to $H'$, we have a partition with $|V_1|=r$ such that
\begin{equation} \label{eq:Hprime-n}
|E(H') \Delta E(K^3(V_1,V_2,V_3))|\le O(n).    
\end{equation}

Using the same argument as in the proof of Theorem \ref{thm:main}, we obtain $|Y|=|Y \cap X_1|+|Y \cap X_2| \leq O(1)+0=O(1)$ and thus
\begin{equation} \label{eq:bt-n}
 |E(H')| \geq |E(H)|-O(1) \geq f(n,m)-t-O(1).   
\end{equation}

If $t$ is a constant at least one, then the desired bound follows by first applying Lemma~\ref{lem:3-core-sta} to $H'$ and then adding back the deleted edges which are at most $O(1)=O(t)$. Thus, we must have $t<1$. Applying Lemma~\ref{lem:3-core-sta} to $H'$, we have a partition with $|V_1|=r$ such that
\begin{equation} \label{eq:Hprime-1}
|E(H') \Delta E(K^3(V_1,V_2,V_3))|\le O(1).    
\end{equation}

Using the same argument as in the proof of Theorem \ref{thm:main}, we obtain $|Y|=|Y \cap X_1|+|Y \cap X_2|=0$ and thus
\begin{equation} \label{eq:bt-n}
 |E(H')| \geq |E(H)|-O(1) \geq f(n,m)-t.   
\end{equation}
Since both $|E(H)|$ and $f(n,m)$ are integers,
Theorem~\ref{thm:main} implies that $H$ is extremal. Thus the desired conclusion holds with zero error.

Moreover, if $t=o(n^2)$, then we observe that $|Y \cap X_2|=0$. This implies $H\subseteq K^3(V_1,V_2,V_3)$ when $t=o(n^2)$. This completes the proof.
\end{proof}

\section{Proofs of Theorems~\ref{thm:main-blow-up} and~\ref{thm:sta-blow-up}}
Recall that $F_5^t$ is the hypergraph on the vertex set $\{1,2,\ldots,5+t\}$ and edge set $\{123,124\} \cup \{34k: 5 \leq k \leq 5+t\}$. Recall that $F_2^t$ is the hypergraph on the vertex set $\{1,2,\ldots,3+t\}$ and edge set $ \{12k: 3 \leq k \leq 3+t\}$. Given a partition $V(H)=V_1 \cup V_2 \cup V_3$, write $N_{V_i}(u,v)=N_H(u,v) \cap V_i$ and $d_{V_i}(u,v)=|N_{V_i}(u,v)|$. Let $H(V_1,V_2,V_3)=H \cap K^3(V_1,V_2,V_3)$.

 Balogh, Clemen, and Luo \cite{BaloghClemenLuo} proved the following lemma.
\begin{lemma}[\cite{BaloghClemenLuo}] \label{lem:F2t-function}
For every integer $t\ge 0$ and sufficiently large $n$, we have
\[
\frac{t}{3}\binom{n}{2}
\ge
\ex(n,F_2^t)
\ge
\begin{cases}
\displaystyle
\frac{t}{3}\binom{n}{2}
-\frac{n}{3}
-\frac{t^2}{6}
-3t,
& \text{if $t$ is odd and $n$ is even},\\[1.9ex]
\displaystyle
\frac{t}{3}\binom{n}{2}
-\frac{8t}{3},
& \text{if $t$ is odd and $n$ is odd},\\[1.9ex]
\displaystyle
\frac{t}{3}\binom{n}{2}
-\frac{2t}{3},
& \text{if $t$ is even}.
\end{cases}
\] 
\end{lemma}

Let $|V_1|=m$, $|V_2|=a_2$, $|V_3|=a_3$ where $a_2+a_3=n-m$. For fixed $m>n/3$, define the function
\[g(n, a_2, a_3)=ma_2a_3+\ex \left(m,F^t_2 \right)+\ex \left(a_2,F^t_2 \right)+\ex \left(a_3,F^t_2 \right).\]
We will show that $g$ achieves its maximum when the two parts $V_2, V_3$ are as equally as possible, i.e.,
\[g(n,a_2,a_3) \leq g(n,\lceil (n-m)/2\rceil,\lfloor (n-m)/2\rfloor).\]
Without loss of generality, assume $a_2 \geq a_3+2$. We now show that shifting one vertex from the $V_2$ to $V_3$ increases $g$. Let $|V'_2|=a_2-1$ and $|V'_3|=a_3+1$. Then 

\begin{align*}
g(n,a_2-1,a_3+1)-g(n,a_2,a_3)
={}&m(a_2-a_3-1)
+\ex\left(a_2-1,F^t_2\right)-\ex\left(a_2,F^t_2\right)\\
&+\ex\left(a_3+1,F^t_2\right)-\ex\left(a_3,F^t_2\right).
\end{align*}

By Lemma~\ref{lem:F2t-function}, we have 
\begin{align*}
\ex\left(a_2-1,F^t_2\right)-\ex\left(a_2,F^t_2\right) &\geq -\frac{t}{3}\binom{a_2}{2}+\frac{t}{3}\binom{a_2-1}{2}-\frac{a_2-1}{3}-\frac{t^2}{6}-3t\\
&=-\frac{t(a_2-1)}{3}-\frac{a_2-1}{3}-\frac{t^2}{6}-3t.
\end{align*}

\begin{align*}
 \ex \left(a_3+1,F^t_2 \right)-\ex \left(a_3,F^t_2 \right) &\geq \frac{t}{3}\binom{a_3+1}{2}-\frac{a_3+1}{3}-\frac{t^2}{6}-3t- \frac{t}{3}\binom{a_3}{2}  \\
 &=\frac{ta_3}{3}-\frac{a_3+1}{3}-\frac{t^2}{6}-3t.
\end{align*}

Therefore,
\begin{align*}
g(n,a_2-1,a_3+1)-g(n,a_2,a_3)
\ge{}&
m(a_2-a_3-1)
+\frac{t(a_3-a_2+1)}3\\
&
-\frac{n-m}{3}
-\frac{t^2}{3}
-6t\\
={}&
\left(m-\frac t3\right)(a_2-a_3-1)
-\frac{n-m}{3}
-\frac{t^2}{3}
-6t.
\end{align*}

Since $a_2-a_3\ge2$, we have
$a_2-a_3-1\ge1$. Hence
\[
g(n,a_2-1,a_3+1)-g(n,a_2,a_3)
\ge
m-\frac t3
-\frac{n-m}{3}
-\frac{t^2}{3}
-6t.
\]
Using $m>cn$, we obtain
\[
m-\frac{n-m}{3}
=
\frac{4m-n}{3}
\ge
\frac{4c-1}{3}n.
\]
Since $c>\frac13$, the coefficient
$\frac{4c-1}{3}$ is positive. Thus, for sufficiently large $n$,
\[
\frac{4c-1}{3}n
>
\frac t3+\frac{t^2}{3}+6t,
\]
which implies
\[
g(n,a_2-1,a_3+1)>g(n,a_2,a_3).
\]
Repeating this procedure until $V_2$ and $V_3$ are balanced gives \[g(n,a_2,a_3) \leq g(n,\lceil (n-m)/2\rceil,\lfloor (n-m)/2\rfloor).\]
This proves the balancing claim.

Define
\[
\Phi_t(n,m)
=
f(n,m)+\ex \left(m,F^t_2 \right)+\ex \left(\left\lceil\frac{n-m}{2}\right\rceil,F^t_2 \right)+\ex \left(\left\lfloor\frac{n-m}{2}\right\rfloor,F^t_2 \right),
\]
where $f(n,m)=m\left\lceil\frac{n-m}{2}\right\rceil
\left\lfloor\frac{n-m}{2}\right\rfloor$. We next show that the function $\Phi_t(n,m)$ is nonincreasing in $m$.

It suffices to show that $\Phi_t(n,m+1)\le \Phi_t(n,m)$. Set $s=n-m$. We consider two parity cases.

\medskip
\noindent\textbf{Case 1: $s=2r$ is even.}
Then
\[
\Phi_t(n,m)=mr^2+\ex(m,F_2^t)+2\ex(r,F_2^t),
\]
while
\[
\Phi_t(n,m+1)=(m+1)r(r-1)+\ex(m+1,F_2^t)+\ex(r-1,F_2^t)+\ex(r,F_2^t).
\]
Hence
\begin{align*}
\Phi_t(n,m+1)-\Phi_t(n,m)
&=
-r(m-r+1)
+\bigl(\ex(m+1,F_2^t)-\ex(m,F_2^t)\bigr) \\
&\qquad
+\bigl(\ex(r-1,F_2^t)-\ex(r,F_2^t)\bigr).
\end{align*}

By Lemma~\ref{lem:F2t-function}, we have
\[
\ex(m+1,F_2^t)-\ex(m,F_2^t)\le \frac{(t+1)m}{3}+\frac{t^2}{6}+3t.
\]
and
\[
\ex(r-1,F_2^t)-\ex(r,F_2^t)\le -\frac{t(r-1)}{3}
+\frac{r}{3}
+\frac{t^2}{6}
+3t.
\]
Therefore, for sufficiently large $n$
\begin{align*}
\Phi_t(n,m+1)-\Phi_t(n,m)
&\le
-(m-r+1)\left(r-\frac t3\right)
+\frac{m+r}{3}
+\frac{t^2}{3}
+6t \leq 0.
\end{align*}

\medskip
\noindent\textbf{Case 2: $s=2r+1$ is odd.}
Then
\[
\Phi_t(n,m)=mr(r+1)+\ex(m,F_2^t)+\ex(r,F_2^t)+\ex(r+1,F_2^t),
\]
and
\[
\Phi_t(n,m+1)=(m+1)r^2+\ex(m+1,F_2^t)+2\ex(r,F_2^t).
\]
So
\begin{align*}
\Phi_t(n,m+1)-\Phi_t(n,m)
&=
-r(m-r)
+\bigl(\ex(m+1,F_2^t)-\ex(m,F_2^t)\bigr) \\
&\qquad
+\bigl(\ex(r,F_2^t)-\ex(r+1,F_2^t)\bigr).
\end{align*}
Using Lemma~\ref{lem:F2t-function}, for sufficiently large $n$, we obtain
\begin{align*}
\Phi_t(n,m+1)-\Phi_t(n,m)
&\le
-(m-r)\left(r-\frac t3\right)
+\frac{m+r}{3}
+\frac{t^2}{3}
+6t \leq 0.
\end{align*}
This proves that $\Phi_t(n,m)$ is nonincreasing in $m$.

\begin{proof}[Proof of Theorem~\ref{thm:main-blow-up}]

The construction in the statement gives the lower bound: take
$K^3_{m,\lceil (n-m)/2\rceil,\lfloor (n-m)/2\rfloor}$ and place an
extremal $F_2^t$-free $3$-graph inside each part.

For the upper bound, let $H$ be an $n$-vertex $F_5^t$-free $3$-graph with
$\Delta_2(H)=r\ge m$. Suppose $|E(H)|\geq \Phi_t(n,m)$ otherwise there is nothing to prove. Choose a pair of vertices $a,b$ satisfying $d_H(a,b)=r \geq m$, and let $M:=N_H(a,b)$. Then $|M|=r \geq m$. Let $H'$ be the $(t+3)$-core of $H$, obtained by repeatedly deleting every edge containing a pair of codegree at most $t+2$. Since at most $t+2$ edges are removed for each pair of vertices,
\[
|E(H')|
\ge
|E(H)|-(t+2)\binom{n}{2}>\Phi_t(n,m)-(t+2)\binom{n}{2}.
\]

\begin{claim} \label{cla:indep}
The set $M$ is strongly independent in $H'$.
\end{claim}

\begin{proof}
Suppose otherwise. Then there exist vertices $c,d\in M$ such that $cd\in E(\partial(H'))$. Thus, we have $d_{H}(c,d)\ge t+3$. Therefore, there exists $(t+1)$ vertices $f_1, \ldots f_{t+1} \in N_H(a,b)$ such that $f_1, \ldots, f_{t+1} \notin\{a,b\}$.  Then $\{abc,abd\} \cup \{cdf_i: 1 \leq i \leq t+1\}$ forms a copy of $F^t_5$, a contradiction.
\end{proof}

\begin{claim}
$H'$ is $F_5$-free.
\end{claim}
We prove the claim. Suppose $H'$ contains a copy of $F_5$, denoted by $\{xyu,xyv,uvw\}$. Since $H'$ is the $(t+3)$-core, then $d_H(u,v) \geq t+3$. Then there exist $t+1$ vertices $\{w_i: 1 \leq i \leq t+1\} \in N_H(u,v) \setminus \{x,y\}$. Then $\{xyu,xyv\} \cup \{uvw_i: 1 \leq i \leq t+1\}$ forms a copy of $F^t_5$. A contradiction.

Since $H'$ is $F_5$-free, we have $|E(H')|\leq f(n,m)$ by Theorem~\ref{thm:main}. Thus,
\begin{equation} \label{eq:F}
|E(H')| \geq \Phi_t(n,m)-(t+2)\binom{n}{2} \geq f(n,m)-O(n^2),    
\end{equation}
Applying Lemma~\ref{lem:3-core-sta} to $H'$, we have a partition with $|V_1|=r$ such that
\begin{equation}\label{eq:Hprime-n2-blowup}
|E(H')\,\Delta\,E(K^3(V_1,V_2,V_3))|
\le O(n^2),    
\end{equation}

First, we show that 
\begin{equation} \label{eq:bt-n2-blowup}
 |E(H')| \geq f(n,m)-O(n). 
\end{equation}

Let $i,j,k \in [3]$. Define the following sets
\[X_1=\{(u,v): u \in V_i, v\in V_j \},\]
\[X_2=\{(u,v):u,v \in V_i \ \mbox{and} \ uvw \in E(H) \ \mbox{for some $w \in V_j$}\},\]
\[X_3=\{(u,v):u,v \in V_i, N(u,v) \subseteq V_i \ \mbox{and} \ d_{V_i}(u,v) \geq t+1\},\]
\[X_4=\{(u,v):u,v \in V_i, N(u,v) \subseteq V_i \ \mbox{and} \ d_{V_i}(u,v) \leq t\}.\]

We say a triple $u,v,w$ is {\em missing} if $e=uvw \in E(K^3(V_1,V_2,V_3)) \setminus E(H')$. By \eqref{eq:Hprime-n2-blowup}, $H'$ has at most $O(n^2)$ missing triples. 

Let $Y$ be the set of pairs $(u,v)$ with $d_H(u,v) \in \{1,\ldots,t+2\}$. We observe that
\[|Y \cap X_4|t \leq \ex \left(m,F^t_2 \right)+\ex \left(\left\lceil\frac{n-m}{2}\right\rceil,F^t_2 \right)+\ex \left(\left\lfloor\frac{n-m}{2}\right\rfloor,F^t_2 \right).\]
If $(u,v) \in Y \cap X_1$, then $(u,v)$ is contained in at least $O(n)$ missing triples and thus $|Y \cap X_1| \leq O(n)$.

Consider some pair $(u,v)\in X_2$ where $uvw \in E(H)$, $u,v \in V_i$ and $w \in V_j$. 
We observe that 
\begin{equation} \label{eq:d(v,z)}
\mbox{if} \ z \in N_{V_k}(u,w), \ \mbox{then} \ d_{V_j}(v,z) \leq t+1.     
\end{equation}

If $d_{V_j}(v,z) \geq t+2$, then there exist $t+1$ vertices $\{w_i: 1 \leq i \leq t+1\} \in N_{V_j}(v,z) \setminus \{w\}$. Then $\{uwv,uwz\} \cup \{vzw_i: 1 \leq i \leq t+1\}$ forms a copy of $F^t_5$. A contradiction. Hence, \eqref{eq:d(v,z)} follows. Let \[S=\{v: \exists u \ \mbox{such that} \ (u,v) \in X_2 \ \mbox{and} \ d_{V_k}(u,w) \geq O(n)\}.\]

Let $v \in S$. By \eqref{eq:d(v,z)}, for each $z \in N_{V_k}(u,w)$, we have $d_{V_j}(v,z) \leq t+1$. Thus, $v$ is contained in at least $d_{V_k}(u,w)(|V_j|-t-1) \geq O(n^2)$ missing triples. Since $H'$ has at most $O(n^2)$ missing triples, $|S| \leq O(1)$. We further partition $X_2=X^1_2 \cup X^2_2$:
\[X^1_2=\{(u,v) \in X_2: \mbox{at least one of $u,v \in S$}\},\]
\[X^2_2=\{(u,v) \in X_2: u,v \notin S\}.\]
Since $|S| \leq O(1)$, we have $|X^1_2|\leq O(n)$. If $(u,v) \in X^2_2$, then we have both of $d_{V_k}(u,w), d_{V_k}(v,w) \leq o(n)$ and thus each of $u,v$ is contained in at least $|V_k|-o(n) \geq O(n)$ missing triples. Since $H'$ has at most $O(n^2)$ missing triples, there are at most $O(n)$ such vertices. Thus, $|X^2_2|\leq O(n)$. Therefore $|X_2|\leq O(n)$ and so $|Y \cap X_2| \leq O(n)$. Let $(u,v) \in X_3$ where $u,v \in V_i$. Then there exist $t+1$ vertices $\{w_i: 1 \leq i \leq t+1\} \subseteq N_{V_i}(u,v)$. For any $w \in V_j$ and $z \in V_k$, if $wzu \in E(H)$, then $wzv \notin E(H)$. If not, then $\{wzu, wzv\} \cup \{uvw_i: 1 \leq i \leq t+1\}$ forms a copy of $F^t_5$. Thus, at least one of $u,v$ is contained in at least $|V_j||V_k|/2 \geq O(n^2)$ missing triples. We call a vertex {\em bad} if it is contained in at least $O(n^2)$ missing triples. Since $H'$ has at most $O(n^2)$ missing triples, the total number of bad vertices is at most $O(1)$. We observe that if $(u,v) \in X_3$, then at least one of $u,v$ is bad. Therefore $|X_3|\leq O(n)$ and so $|Y \cap X_3| \leq O(n)$. Hence, \eqref{eq:bt-n2-blowup} follows. Applying Lemma \ref{lem:3-core-sta} to $H'$, we have a partition with $|V_1|=r$ such that
\begin{equation}\label{eq:Hprime-n-blowup}
|E(H')\,\Delta\,E(K^3(V_1,V_2,V_3))|
\le O(n).    
\end{equation}

We now show that 
\begin{equation} \label{eq:bt-n-blowup}
 |E(H')| \geq f(n,m)-O(1). 
\end{equation}

There are at most $O(n)$ missing triples in $H'$. If $(u,v) \in Y \cap X_1$, then $(u,v)$ is contained in at least $O(n)$ missing triples and thus $|Y \cap X_1| \leq O(1)$. Since each $v \in S$ is contained in at least $O(n^2)$ missing triples, we must have $|S|=0$ and thus $|X^1_2|=0$. If $(u,v) \in X^2_2$, then both $u$ and $v$ are contained in at least $O(n)$ missing triples so $|X^2_2|\leq O(1)$. Therefore, $|Y \cap X_2| \leq |X_2| \leq O(1)$. Since each bad vertex is contained in at least $O(n^2)$ missing triples, there is no bad vertex. Since every pair of vertices in $X_3$ contains a bad vertex, we have $|X_3|=0$ and so $|Y \cap X_3|=0$. Hence, \eqref{eq:bt-n-blowup} follows. Applying Lemma \ref{lem:3-core-sta} to $H'$, we have a partition with $|V_1|=r$ such that
\begin{equation}\label{eq:Hprime-1-blowup}
|E(H')\,\Delta\,E(K^3(V_1,V_2,V_3))|
\le O(1).  
\end{equation}

Since $H'$ has at most $O(1)$ missing triples, using similar arguments as above, we have $|Y \cap X_1|=|X_2|=|X_3|=0$. Therefore, we have $Y \subseteq X_4$. This implies
\[|E(H \setminus H')|\leq \ex \left(|V_1|,F^t_2 \right)+\ex \left(|V_2|,F^t_2 \right)+\ex \left(|V_3|,F^t_2 \right).\]
Since $|X_2|=|X_3|=0$, we have $H' \subseteq K^3(V_1,V_2,V_3)$. Then $|E(H')|\leq |V_1||V_2||V_3|$. Hence,
\[|E(H)|\leq |V_1||V_2||V_3|+\ex \left(|V_1|,F^t_2 \right)+\ex \left(|V_2|,F^t_2 \right)+\ex \left(|V_3|,F^t_2\right) \leq \Phi_t(n,m).\]

The maximum is achieved by $K^3_{\,m,\,
\left\lceil \frac{n-m}{2}\right\rceil,\,
\left\lfloor \frac{n-m}{2}\right\rfloor}$ with a copy of $H(|V_i|,t) \in \mathcal{H}(|V_i|,t)$ inside each part.

This completes the proof of Theorem~\ref{thm:main-blow-up}.
\end{proof}

\begin{proof}[Proof of Theorem~\ref{thm:sta-blow-up}]
Choose a pair of vertices $a,b$ with $\Delta_2(G)=d_H(a,b)=r\ge m$, and let $M= N_H(a,b)$.
Let $H'$ be the $(t+3)$-core of $H$, obtained by repeatedly deleting every
edge containing a pair of codegree at most $t+2$. Since at most
$(t+2)\binom n2$ edges are deleted, we have
\begin{equation}\label{eq:stab-blowup-core-first}
        |E(H')|\ge \Phi_t(n,m)-\tau-(t+2)\binom n2 \geq f(n,m)-\tau-O(n^2).
\end{equation}
As in the proof of Theorem~\ref{thm:main-blow-up}, the $(t+3)$-core $H'$ is
$F_5$-free and $M$ is strongly independent.

Let $E_i$ be the set of edges that contains a pair of vertices $u,v$ such that $d_H(u,v) \leq t+2$ and $(u,v) \in X_i$ for $i \in [4]$.
\begin{claim} \label{cla:inside edges}
For a given partition $V(H)=V_1\cup V_2 \cup V_3$, let $J_i$ be the extremal $F_2^t$-free subhypergraph of $H[V_i]$. We have 
\[|E(H[V_1] \cup H[V_2] \cup H[V_3]) \Delta E(J_1 \cup J_2 \cup J_3)|\leq |E_1 \cup E_2 \cup E_3|+\tau. \]
\end{claim}

\begin{proof}
Suppose otherwise. We have 
\[|E_4| \leq \ex \left(|V_1|,F^t_2 \right)+\ex \left(|V_2|,F^t_2 \right)+\ex \left(|V_3|,F^t_2 \right)-\tau-|E_1 \cup E_2 \cup E_3|.\]
Then
\[|E(H')|\geq f(n,m)+\ex \left(|V_1|,F^t_2 \right)+\ex \left(|V_2|,F^t_2 \right)+\ex \left(|V_3|,F^t_2 \right)-\tau-|E_1 \cup E_2 \cup E_3|-|E_4| \geq f(n,m),\]   
contradicting to $H'$ being $F_5$-free.
\end{proof}

If $\tau$ is at least a sufficiently large constant multiple of $n^2$, then applying Lemma~\ref{lem:3-core-sta} to $H'$ and then adding back the deleted edges $E_1 \cup E_2 \cup E_3$ which are at most $O(n^2)=O(\tau)$, we have \[|E(H(V_1,V_2,V_3)) \Delta E(K^3(V_1,V_2,V_3))|
        \le O(\tau).\] 
Since $|E(H[V_1] \cup H[V_2] \cup H[V_3]) \Delta E(J_1 \cup J_2 \cup J_3)|\leq O(\tau)$ by Claim \ref{cla:inside edges}, the desired bound follows.        
        Thus, we may assume $\tau=o(n^2)$. Applying Lemma~\ref{lem:3-core-sta} to $H'$, we have 
\begin{equation}\label{eq:stab-blowup-close-n2}
        |E(H')\Delta E(K^3(V_1,V_2,V_3))|
        \le O(n^2).
\end{equation}

Let $Y$ be the set of pairs $(u,v)$ with $d_H(u,v)\in \{1, \ldots, t+2\}$. Using the same argument as in the proof of Theorem \ref{thm:main-blow-up}, we obtain $|Y \cap X_1|+|Y \cap X_2|+|Y \cap X_3| \leq O(n)$ and thus
\begin{equation}\label{eq:stab-blowup-core-second}
        |E(H')|\ge f(n,m)-\tau-O(n).
\end{equation}
If $\tau$ is at least a sufficiently large constant multiple of $n$, then applying Lemma~\ref{lem:3-core-sta} to $H'$ and then adding back the deleted edges $E_1 \cup E_2 \cup E_3$ which are at most $O(n)=O(\tau)$, we have \[|E(H(V_1,V_2,V_3)) \Delta E(K^3(V_1,V_2,V_3))|
        \le O(\tau).\] 
Since $|E(H[V_1] \cup H[V_2] \cup H[V_3]) \Delta E(J_1 \cup J_2 \cup J_3)|\leq O(\tau)$ by Claim \ref{cla:inside edges}, the desired bound follows. Thus, we may assume $\tau=o(n)$. Applying Lemma~\ref{lem:3-core-sta}  to $H'$, we have
\begin{equation}\label{eq:stab-blowup-close-n}
        |E(H')\Delta E(K^3(V_1,V_2,V_3))|
        \le O(n).
\end{equation}
Using the same argument as in the proof of Theorem \ref{thm:main-blow-up}, we obtain $|Y \cap X_1|+|Y \cap X_2|+|Y \cap X_3| \leq O(1)+O(1)+0=O(1)$ and thus
\begin{equation}\label{eq:stab-blowup-core-third}
        |E(H')|\ge f(n,m)-\tau-O(1).
\end{equation}
If $\tau$ is a constant at least one, then applying Lemma~\ref{lem:3-core-sta} to $H'$ and then adding back the deleted edges $E_1 \cup E_2 \cup E_3$ which are at most $O(1)=O(\tau)$, we have \[|E(H(V_1,V_2,V_3)) \Delta E(K^3(V_1,V_2,V_3))|
        \le O(\tau).\] 
Since $|E(H[V_1] \cup H[V_2] \cup H[V_3]) \Delta E(J_1 \cup J_2 \cup J_3)|\leq O(\tau)$ by Claim \ref{cla:inside edges}, the desired bound follows. Thus, we must have $\tau<1$. Applying Lemma~\ref{lem:3-core-sta} to $H'$, we have 
\begin{equation}\label{eq:stab-blowup-final-crossing}
        |E(H')\Delta E(K^3(V_1,V_2,V_3))|
        \le O(1).
\end{equation}
Using the same argument as in the proof of Theorem \ref{thm:main-blow-up}, we obtain $|Y \cap X_1|+|Y \cap X_2|+|Y \cap X_3|=0$ and thus
\begin{equation}\label{eq:stab-blowup-core-third}
        |E(H')|\ge f(n,m)-\tau.
\end{equation}
Since both $|E(H')|$ and $f(n,m)$ are integers,
Theorem~\ref{thm:main} implies that $H'$ is extremal i.e., $H'=K^3(V_1,V_2,V_3)$. By Claim \ref{cla:inside edges}, we have $H[V_1] \cup H[V_2] \cup H[V_3] =J_1 \cup J_2 \cup J_3$. Thus the desired conclusion holds with zero error.

\end{proof}

\medskip
\noindent\textbf{Acknowledgements.} ChatGPT 5.5 were used to improve the presentation of the manuscript; 
the main concepts, ideas, and proofs were developed independently by the authors
and we take full responsibility for the contents of the paper.

\end{document}